\documentclass[12pt]{amsart}
\usepackage{amsmath,amssymb,latexsym,cancel,rotating}

\numberwithin{equation}{section}

\newtheorem{theorem}{Theorem}[section]
\newtheorem{proposition}[theorem]{Proposition}

\newtheorem{corollary}[theorem]{Corollary}
\newtheorem{lemma}[theorem]{Lemma}

\theoremstyle{definition}

\newtheorem{remark}[theorem]{Remark}

\newtheorem{definition}[theorem]{Definition}

\input xy
\xyoption{poly}
\xyoption{2cell}
\xyoption{all}

\def\ZZ{\mathbb{Z}}

\def\QQ{\mathbb{Q}}

\def\Acal{\mathcal{A}}

\def\Fcal{\mathcal{F}}

\def\Tcal{\mathcal{T}}

\def\dimb{\underbar{\text{dim}}}

\def\half{\frac{1}{2}}

\renewcommand{\eqref}[1]{{\rm (\ref{#1})}}

\begin{document}

\title[Bases of  the quantum cluster algebra of the Kronecker quiver]
{Bases of  the quantum cluster algebra of the Kronecker quiver}

\author{Ming Ding and Fan Xu}
\address{Institute for advanced study\\
Tsinghua University\\
Beijing 100084, P.~R.~China} \email{m-ding04@mails.tsinghua.edu.cn
(M.Ding)}
\address{Department of Mathematical Sciences\\
Tsinghua University\\
Beijing 100084, P.~R.~China} \email{fanxu@mail.tsinghua.edu.cn
(F.Xu)}


\thanks{Fan Xu was partially supported by
the Ph.D. Programs Foundation of Ministry of Education of China (No.
200800030058)}



\maketitle

\begin{abstract} We construct bar-invariant $\mathbb{Z}[q^{\pm \frac{1}{2}}]-$bases
of the quantum cluster algebra of the Kronecker quiver which are
quantum analogues of the canonical basis, semicanonical basis and
dual semicanonical basis of the cluster algebra of the Kronecker
quiver in the sense of \cite{sherzel},\cite{calzel} and \cite{gls}
respectively. As a byproduct, we prove the positivity of the
elements in these bases.
\end{abstract}


\section{Introduction}
Cluster algebras were introduced by S. Fomin and A. Zelevinsky
\cite{ca1}\cite{ca2} in order to study total positivity in algebraic
groups and the specialization  of canonical bases of quantum groups
at $q=1$. The study of $\mathbb{Z}$-bases of cluster algebras is
important. There are many results involving the construction of
$\mathbb{Z}$-bases of cluster algebras (for example, see
\cite{sherzel} and \cite{calzel} for cluster algebras of rank 2,
\cite{CK2005} for finite type, \cite{Dup} for type $\widetilde{A}$,
\cite{gc} for $\widetilde{A}_{2}^{(1)}$, \cite{DXX} for affine type
and \cite{gls} for $Q$ without oriented cycles). As a quantum analog
of cluster algebras, quantum cluster algebras were defined by A.
Berenstein and A. Zelevinsky in \cite{berzel} in order to study
canonical bases. A quantum cluster algebra   is generated by a set
of generators called \textit{cluster variables} inside an ambient
skew-field $\mathcal{F}$. Under the specialization $q=1,$ the
quantum cluster algebras are exactly cluster algebras.

Naturally, one may hope to construct $\mathbb{Z}[q^{\pm
\frac{1}{2}}]-$bases for quantum cluster algebras and further
quantum analogues of bases of the corresponding cluster algebras. In
this short note, we deal with the case of the quantum cluster
algebra of the Kronecker quiver and construct various bar-invariant
$\mathbb{Z}[q^{\pm \frac{1}{2}}]-$bases by applying the
$q$-deformation of the Caldero-Chapoton formula defined in \cite
{rupel} and  the method in \cite{sherzel}. Under the specialization
$q=1$, these $\mathbb{Z}[q^{\pm \frac{1}{2}}]-$bases are exactly the
canonical basis, semicanonical basis and dual semicanonical basis of
the corresponding cluster algebra in the sense of
\cite{sherzel},\cite{calzel} and \cite{gls} respectively. As a
byproduct, we prove the positivity of the elements in these bases.

Recently, in \cite{lampe}, the author attached to certain element
$w$ in Weyl group a subalgebra $U_q^{+}(w)$ of the positive part
$U_q(n)$ of the universal enveloping algebra of a Kac-Moody Lie
algebra of type $\widetilde{A}_1^{(1)}$. The author proved that
$U_q^{+}(w)$ is a quantum cluster algebra in the sense of
Berenstein-Zelevinsky and gave explicit formulae for the cluster
variables. Note that the cluster variables are some elements of
$q$-deformation of dual canonical basis elements of $U_q^{+}(w)$.
However it is not clear whether cluster monomials belong to the dual
canonical basis. Thus comparing these $\mathbb{Z}[q^{\pm
\frac{1}{2}}]-$bases constructed in this note with the dual
canonical basis of $U_q^{+}(w)$ becomes an interesting thing.

\section{Preliminaries}
\subsection{Quantum cluster algebras}We begin with  some of the terminology related to quantum
cluster algebras.  One can refer to \cite{berzel} for more details.
Let $L$ be a lattice of rank $m$ and $\Lambda:L\times L\to \ZZ$ a
skew-symmetric bilinear form. We will need a formal variable $q$ and
consider the ring of integer Laurent polynomials $\ZZ[q^{\pm1/2}]$.
Define the \textit{based quantum torus} associated to the pair
$(L,\Lambda)$ to be the $\ZZ[q^{\pm1/2}]$-algebra $\mathcal{T}$ with
a distinguished $\ZZ[q^{\pm1/2}]$-basis $\{X^e: e\in L\}$ and the
multiplication given by
\[X^eX^f=q^{\Lambda(e,f)/2}X^{e+f}.\]
It is easy to see  that $\Tcal$ is associative and the basis
elements satisfy the following relations:
\[X^eX^f=q^{\Lambda(e,f)}X^fX^e,\ X^0=1,\ (X^e)^{-1}=X^{-e}.\]  It is known that $\Tcal$ is an Ore domain, i.e.,   is contained in its
skew-field of fractions $\Fcal$.  The quantum cluster algebra
 will be defined as a
$\ZZ[q^{\pm1/2}]$-subalgebra of $\Fcal$.

A \textit{toric frame} in $\Fcal$ is a map $M: \ZZ^m\to \Fcal
\setminus \{0\}$ of the form \[M({\bf c})=\varphi(X^{\eta({\bf
c})})\] where $\varphi$ is an automorphism of $\Fcal$ and $\eta:
\ZZ^m\to L$ is an  isomorphism of lattices.  By the definition, the
elements $M({\bf c})$ form a $\ZZ[q^{\pm1/2}]$-basis of the based
quantum torus $\Tcal_M:=\varphi(\Tcal)$ and satisfy the following
relations:
\[M({\bf c})M({\bf d})=q^{\Lambda_M({\bf c},{\bf d})/2}M({\bf c}+{\bf d}),\
M({\bf c})M({\bf d})=q^{\Lambda_M({\bf c},{\bf d})}M({\bf d})M({\bf c}),\]
\[ M({\bf 0})=1,\ M({\bf c})^{-1}=M(-{\bf c}),\]
where $\Lambda_M$ is the skew-symmetric bilinear form on $\ZZ^m$
obtained from the lattice isomorphism $\eta$.  Let $\Lambda_M$ also
denote the skew-symmetric $m\times m$ matrix defined by
$\lambda_{ij}=\Lambda_M(e_i,e_j)$ where $\{e_1, \ldots, e_m\}$ is
the standard basis of $\ZZ^m$.  Given a toric frame $M$, let
$X_i=M(e_i)$.  Then we have
$$\Tcal_M=\QQ(q^{1/2})\langle X_1^{\pm 1}, \ldots,
X_m^{\pm1}:X_iX_j=q^{\lambda_{ij}}X_jX_i\rangle.$$  An easy
computation shows that:
\[M({\bf c})=q^{\frac{1}{2}\sum_{i<j}
c_ic_j\lambda_{ji}}X_1^{c_1}X_2^{c_2}\cdots X_m^{c_m}=:X^{({\bf c})}
\ \ \ ({\bf c}\in\ZZ^m).\]

Let $\Lambda$ be an $m\times m$ skew-symmetric matrix and let
$\tilde{B}$ be an $m\times n$ matrix, $n\le m$.  We call the pair
$(\Lambda, \tilde{B})$ \textit{compatible} if
$\tilde{B}^T\Lambda=(D|0)$ is an $n\times m$ matrix with
$D=diag(d_1,\cdots,d_n)$ where $d_i\in \mathbb{N}$ for $1\leq i\leq
n$. The pair $(M,\tilde{B})$ is called a \textit{quantum seed} if
the pair $(\Lambda_M, \tilde{B})$ is compatible.  Define the
$m\times m$ matrix $E=(e_{ij})$ by
\[e_{ij}=\begin{cases}
\delta_{ij} & \text{if $j\ne k$;}\\
-1 & \text{if $i=j=k$;}\\
max(0,-b_{ik}) & \text{if $i\ne j = k$.}
\end{cases}
\]
For $n,k\in\ZZ$, $k\ge0$, denote ${n\brack
k}_q=\frac{(q^n-q^{-n})\cdots(q^{n-r+1}-q^{-n+r-1})}{(q^r-q^{-r})\cdots(q-q^{-1})}$.
Let ${\bf c}=(c_1,\ldots,c_m)\in\ZZ^m$ with $c_{k}\geq 0$.  Define
the toric frame $M': \ZZ^m\to \Fcal \setminus \{0\}$ as follows:
\begin{equation}\label{eq:cl_exp}M'({\bf c})=\sum^{c_k}_{p=0} {c_k \brack p}_{q^{d_k/2}} M(E{\bf c}+p{\bf b}^k),\ \ M'({\bf -c})=M'({\bf c})^{-1}.\end{equation}
where the vector ${\bf b}^k\in\ZZ^m$ is the $k-$th column of
$\tilde{B}$.    Then the quantum seed $(M',\tilde{B}')$ is defined
to be the mutation of $(M,\tilde{B})$ in direction $k$. We say that
two quantum seeds are mutation-equivalent if they can ba obtained
from each other by a sequence of mutations. Let
$\mathcal{C}=\{M'(e_i): i\in[1,n]\}$ where $(M',\tilde{B}')$ is
mutation-equivalent to $(M,\tilde{B})$. The elements of
$\mathcal{C}$ are called \textit{cluster variables}. Let
$\mathcal{P}=\{M(e_i): i\in[n+1,m]\}$ and  the elements of
$\mathcal{P}$ are called  \emph{coefficients}.  The \textit{quantum
cluster algebra} $\Acal_q(\Lambda_M,\tilde{B})$ is the
$\ZZ[q^{\pm1/2}]$-subalgebra of $\Fcal$ generated by
$\mathcal{C}\cup \mathcal{P}$. We associated with $(M,\tilde{B})$
the  $\ZZ$-linear \emph{bar-involution} on $\Tcal_M$ by setting:
\[\overline{q^{r/2}M({\bf c})}=q^{-r/2}M({\bf c}), \ \  (r\in\ZZ,\ {\bf
c}\in\ZZ^n).\]
 It is easy to show that
$\overline{XY}=\overline{Y}~\overline{X}$ for all $X,Y\in
\Acal_q(\Lambda_M,\tilde{B})$ and that each element of
$\mathcal{C}\cup \mathcal{P}$ is \emph{bar-invariant}.

\subsection{The Kronecker quiver}
 Given a compatible pair $(\Lambda, \widetilde{B})$, we can associate a valued quiver (see \cite[Section 2]{rupel} for more details).  Now we set  $\Lambda=\left(\begin{array}{cc} 0 & 1\\ -1 &
0\end{array}\right)$ and $\widetilde{B}=\left(\begin{array}{cc} 0 & 2\\
-2 & 0\end{array}\right)$. The quiver $Q$ associated to this pair is
the Kronecker quiver: \vspace*{-0.5cm}
\begin{center}
 \setlength{\unitlength}{0.61cm}
 \begin{picture}(5,4)
 \put(0,2){1}\put(0.4,2.2){$\bullet$}
\put(3,2.2){$\bullet$}\put(3.4,2){2} \put(0.8,2.5){\vector(3,0){2}}
\put(0.8,2.2){\vector(3,0){2}}
 \end{picture}
 \end{center}
\vspace*{-1cm}

Let $k$ be a finite field with cardinality $|k|=q^2$. The category
$rep(kQ)$ of finite-dimensional representations can be identified
with the category of mod-$kQ$ of finite-dimensional modules over the
path algebra $kQ.$ It is well-known (see \cite{dlab}) that
indecomposable $kQ$-module contains (up to isomorphism) three
families: the indecomposable regular modules with dimension vector
$(nd_p, nd_p)$ for $p\in \mathbb{P}^1_k$ of degree $d_p$ (in
particular, denoted by $R_p(n)$ for $d_p=1$),  the preprojective
modules with dimension vector $(n-1,n)$ (denoted by $M(n)$)  and the
preinjective modules with dimension vector $(n,n-1)$ (denoted by
$N(n)$). Here $n\in \mathbb{N}$.

For $m\in \mathbb{Z}\setminus\{1, 2\}$, set
\[V(m)=\begin{cases}
N(m-2)& \text{if $m\geq 3$;}\\
M(-m+1) & \text{if $m\leq 0$.}
\end{cases}
\]

Now, let  $\Tcal=\QQ(q^{1/2})\langle X_1^{\pm1}, X_2^{\pm1}:
X_1X_2=qX_2X_1\rangle$ and  ${\mathcal F}$ be the skew field of
fractions of $\Tcal$ and thus the quantum cluster algebra of the
Kronecker quiver $\Acal_q(\Lambda,\tilde{B})$ (denoted by
$\Acal_q(2,2)$ in the following) is the $\QQ(q^{1/2})$-subalgebra of
${\mathcal F}$ generated by the cluster variables $X_k$, $k\in\ZZ$,
defined recursively by
\begin{equation}\label{recurrence}
X_{m-1}X_{m+1}= qX_m^2+1.
\end{equation}
The quantum Laurent phenomenon (\cite{berzel}) implies that each
$X_k$ belongs to the subring of $\Tcal$ generated by $q^{\pm 1/2}$,
$X_1^{\pm1}, X_2^{\pm1}$. The explicit Laurent expansion of each
$X_k$  in $X_1,X_2$  is given in \cite{lampe} and \cite{rupel}.

Let $V$ be a  representation of the Kronecker quiver with dimension
vector $\dimb V = (v_1,v_2)$.  For ${\bf
e}=(e_1,e_2)\in\ZZ_{\ge0}^2$, denote by $Gr_{\bf e}(V)$ the set of
all subrepresentations $M$ of $V$  with $\dimb M={\bf e}$.  In
\cite{rupel}, the author define the element $X_V$ of the quantum
torus $\Tcal$ by
\begin{equation}
\label{eq:XV rank 2} X_V=\sum_{{\bf e}} q^{-\half d_{{\bf
e}}^V}|Gr_{{\bf e}}(V)| X^{(-v_1+2v_2-2e_2,2e_1-v_2)}
\end{equation}
where $d_{{\bf e}}^V=2e_1(v_1-e_1)-2(2e_1-e_2)(v_2-e_2)$. This
formula is called a  $q$-deformation of the Caldero-Chapoton formula
(\cite{caldchap}). Here and the following, we simply write
$X^{\textbf{c}}$ instead of $X^{(\textbf{c})}$ for $\textbf{c}\in
\mathbb{Z}^2.$

\begin{theorem}\cite{rupel}\label{rank2}
For any   $m\in\ZZ\setminus \{1,2\}$, the $m$-th cluster variable
$X_m$ of $\Acal_q(2,2)$ equals $X_{V(m)}$.
\end{theorem}

\section{Bases of the quantum cluster algebra $\mathcal{A}_q(2,2)$}
In this section, we will  construct various bar-invariant
$\mathbb{Z}[q^{\pm \frac{1}{2}}]-$bases of quantum cluster algebra
$\mathcal{A}_q(2,2).$ Under the specialization $q=1$, these bases
are just bases of the cluster algebra of the Kronecker quiver.
\begin{definition}\label{def}
For  any $(r_{1},r_{2})$ and $(s_{1},s_{2})\in \mathbb{Z}^{2}$, we
write $(r_{1},r_{2})\preceq (s_{1},s_{2})$ if $r_{i}\leq s_{i}$ for
$1\leq i\leq 2$. Moreover, if there exists some $i$ such that
$r_{i}< s_{i}$, then we write $(r_{1},r_{2})\prec (s_{1},s_{2}).$
\end{definition}

\begin{remark}\label{6}
By the definition of the $q$-deformation of the Caldero-Chapoton
formula and the partial order in Definition \ref{def}, we obtain
that the expansion of $X_{V(m)}$ have a minimal non-zero term
$f(q^{\frac{1}{2}}, q^{-\frac{1}{2}})X^{-\dimb V(m)}$  where
$f(q^{\frac{1}{2}}, q^{-\frac{1}{2}})\in \mathbb{Z}[q^{\frac{1}{2}},
q^{-\frac{1}{2}}]$. In fact, $f(q^{\frac{1}{2}},
q^{-\frac{1}{2}})=1$ by the explicit expansion of $X_{V(m)}$ in
\cite{lampe} and \cite[Proposition 1.1]{rupel}.
\end{remark}
\begin{lemma}\label{1}
Let $R_{p}(1)$ be the indecomposable regular module of degree $1$ as
above. Then $$X_{R_{p}(1)}=X^{(-1,1)}+X^{(1,-1)}+X^{(-1,-1)}.$$
\end{lemma}
\begin{proof}
Note that $R_p(1)$ contains the three submodules: $0,  M(1)$ and
$R_p(1)$. Thus the lemma immediately follows from the
$q$-deformation of the Caldero-Chapoton formula.
\end{proof}

By Lemma \ref{1}, the expression of $X_{R_p(1)}$ is independent of
the choice of  $p\in \mathbb{P}^1_k$ of degree 1. Hence, we set
$$
X_{\delta}:= X_{R_p(1)}.
$$
\begin{definition}
(1)\ The n-th Chebyshev polynomials of the first kind is the
polynomial $F_{n}(x)\in \mathbb{Z}[x]$ defined by
$$F_{0}(x)=1, F_{1}(x)=x, F_{2}(x)=x^2-2,F_{n+1}(x)=F_{n}(x)F_{1}(x)-F_{n-1}(x) \ for \ n\geq 2.$$

(2)\ The n-th Chebyshev polynomials of the second kind is the
polynomial $S_{n}(x)\in \mathbb{Z}[x]$ defined by
$$S_{0}(x)=1, S_{1}(x)=x, S_{2}(x)=x^2-1,S_{n+1}(x)=S_{n}(x)S_{1}(x)-S_{n-1}(x) \ for \ n\geq 2.$$
\end{definition}
It is obvious that $F_{n}(x)=S_{n}(x)-S_{n-2}(x).$ We denote
$z=X_{\delta}, z_{n}=F_{n}(z), s_{n}=S_{n}(z)$ for $n\geq 0$ and
$z_{n}=s_{n}=0$ for $n<0$. Set
$$\mathcal{B}=\{X^{a}_{m}X^{b}_{m+1}|m\in \mathbb{Z}, (a,b)\in \mathbb{Z}^{2}_{\geq
0}\}\cup \{z_{n}|n\in \mathbb{N}\}$$
$$\mathcal{S}=\{X^{a}_{m}X^{b}_{m+1}|m\in \mathbb{Z},(a,b)\in
\mathbb{Z}^{2}_{\geq 0}\}\cup \{s_{n}|n\in \mathbb{N}\}$$
$$\mathcal{D}=\{X^{a}_{m}X^{b}_{m+1}|m\in \mathbb{Z},(a,b)\in
\mathbb{Z}^{2}_{\geq 0}\}\cup \{z^{n}|n\in \mathbb{N}\}$$

\begin{remark}\label{7}
(1)\ It is easy to check that $X^{(r,r)}X^{(s,s)}=X^{(r+s,r+s)}$ for
any $r,s\in \mathbb{Z}$, thus the expansions of $z_{n},s_{n}$ and
$z^{n}$ have a minimal non-zero term $f(q^{\frac{1}{2}},
q^{-\frac{1}{2}})X^{-(n,n)}$ according to the partial order in
Definition \ref{def}.

(2)\ The elements $\textbf{c}$ associated to  these minimal non-zero
terms $f(q^{\frac{1}{2}}, q^{-\frac{1}{2}})X^{\textbf{c}}$ in the
expansion of the elements in the set $\mathcal{B}$ are different
from  each other. Indeed, it is easy to compute
\[\textbf{c}=\begin{cases}
(b, -a)& \text{if $X^{\textbf{c}}=X_0^aX_1^b$;}\\
(a, b)& \text{if $X^{\textbf{c}}=X_1^aX_2^b$;}\\
(-b, a)& \text{if $X^{\textbf{c}}=X_2^aX_3^b$;}\\
a\cdot\dimb V(m)+b\cdot\dimb V(m+1) & \text{if $X^{\textbf{c}}=X_{m}^aX_{m+1}^b$ for $m\neq 0, 1 \mbox{ or }2$;}\\
(n, n) & \text{if $X^{\textbf{c}}=z_n$.}
\end{cases}
\]
We note that there is at most one exceptional module in each
dimension vector.
\end{remark}

Now we define a ring homomorphism of the quantum cluster algebra
$\Acal_q(2,2)$:
$$\sigma_{1}:\ \Acal_q(2,2)\longrightarrow \Acal_q(2,2)$$
which sends $X_m$ to $X_{m+1}$ and $q^{\pm\frac{1}{2}}$ to
$q^{\pm\frac{1}{2}}$. It is obviously an automorphism which
preserves the defining relations. The following lemma is easy but
important.

\begin{lemma}\label{2}
$$\sigma_{1}(X_{\delta})=X_{\delta}.$$
\end{lemma}
\begin{proof}
By Theorem \ref{rank2} and the definition of the $q$-deformation of
the Caldero-Chapoton formula, we have
\begin{eqnarray}
    X_{0}=X_{V(0)} &=& X^{(2,-1)}+ X^{(0,-1)}, \nonumber\\
    X_{3}=X_{V(3)} &=& X^{(-1,2)}+ X^{(-1,0)}, \nonumber\\
    X_{-1}=X_{V(-1)}&=&  X^{(3,-2)}+X^{(-1,-2)}+(q+q^{-1})X^{(1,-2)}+X^{(-1,0)}.\nonumber
\end{eqnarray}
Following these identities and Lemma \ref{1}, one easily confirm the
relations
\begin{equation}\label{imaginary}
X_{\delta}=q^{\frac{1}{2}}(X_0X_3-qX_1X_2)=q^{\frac{1}{2}}(X_{-1}X_2-qX_0X_1).\end{equation}
Thus
$\sigma_{1}(X_{\delta})=\sigma_{1}(q^{\frac{1}{2}}(X_{-1}X_2-qX_0X_1))=q^{\frac{1}{2}}(X_0X_3-qX_1X_2)=X_{\delta}.$
\end{proof}

\begin{lemma}\label{3}
For any $n\in \mathbb{Z},$
$$X_{n}X_{\delta}=q^{-\frac{1}{2}}X_{n-1}+q^{\frac{1}{2}}X_{n+1}.$$
\end{lemma}
\begin{proof}
By an easy computation, we have:
\begin{eqnarray}
    X_{0} &=& X^{(2,-1)}+ X^{(0,-1)} \nonumber\\
    X_{-1}&=&  X^{(3,-2)}+X^{(-1,-2)}+(q+q^{-1})X^{(1,-2)}+X^{(-1,0)}.\nonumber
\end{eqnarray}
Then by Lemma \ref{1}, it is easy to prove
$$X_{0}X_{\delta}=q^{-\frac{1}{2}}X_{-1}+q^{\frac{1}{2}}X_{1}.$$
Thus we can finish the proof by Lemma \ref{2} and applying the
automorphism $\sigma_{1}$.
\end{proof}

\begin{lemma}\label{4}
$$\overline{X_{\delta}}=X_{\delta}.$$
\end{lemma}
\begin{proof}
$\overline{X_{\delta}}=q^{-\frac{1}{2}}(\overline{X_0X_3}-q^{-1}\overline{X_1X_2})=q^{-\frac{1}{2}}(X_3X_0-q^{-1}X_2X_1)=X_{\delta}.$
\end{proof}

\begin{remark}\label{9}
By Lemma \ref{4}, we can verify that
$\overline{z_{n}}=z_{n},\overline{s_{n}}=s_{n}$ and
$\sigma_{1}(z_{n})=z_{n},\sigma_{1}(s_{n})=s_{n}$.
\end{remark}

The following proposition, which can be viewed as the quantum
analogue of \cite[Proposition 5.4]{sherzel},  plays an essential
role to construct $\mathbb{Z}[q^{\pm \frac{1}{2}}]-$bases of the
quantum cluster algebra $\Acal_q(2,2).$
\begin{proposition}\label{5}
(1)\ For $m>n\geq 1$:
\begin{eqnarray}
    z_{n}z_{m} &=& z_{m+n}+z_{m-n} \nonumber\\
    z_{n}z_{n} &=&  z_{2n}+2.\nonumber
\end{eqnarray}
(2)\ $m\geq 1$  and $n\in \mathbb{Z}$:
\begin{eqnarray}
    X_{n}z_{m} &=&  q^{\frac{m}{2}}X_{n+m}+q^{-\frac{m}{2}}X_{n-m}.\nonumber
\end{eqnarray}
(3)\ For $m\geq 0 \ and\ n\in \mathbb{Z}:$
\begin{eqnarray}
   X_{n}X_{n+2m} &=& q^{m}X^{2}_{n+m}+\sum_{l=0}^{m-1}q^{-m+2l+1}\sum_{k=l+1}^m z_{2(m-k)}, \nonumber\\
   X_{n}X_{n+2m+1} &=& q^{m}X_{n+m}X_{n+m+1}+\sum_{l=0}^{m-1}q^{-m+2l+\frac{1}{2}}\sum_{k=l+1}^m z_{2(m-k)+1}.\nonumber
\end{eqnarray}
\end{proposition}
\begin{proof}
 The proof of (1) follows from the inductive relations in the definition of Chebyshev
polynomials.\\
As for (2), we make induction on $m$. If $m=1,$ the equation in (2)
is a direct corollary of Lemma \ref{3}. We assume that (2) holds for
$m\leq k.$ For $m=k+1,$ we have
\begin{eqnarray}
    X_{n}z_{k+1} &=& X_{n}(z_{k}z_{1}-z_{k-1})  \nonumber\\
&=&(q^{\frac{k}{2}}X_{n+k}+q^{-\frac{k}{2}}X_{n-k})z_{1}-(q^{\frac{k-1}{2}}X_{n+k-1}+q^{-\frac{k-1}{2}}X_{n-k+1})\nonumber\\
&=&q^{\frac{k}{2}}X_{n+k}z_{1}+q^{-\frac{k}{2}}X_{n-k}z_{1}-(q^{\frac{k-1}{2}}X_{n+k-1}+q^{-\frac{k-1}{2}}X_{n-k+1})\nonumber\\
&=&q^{\frac{k}{2}}(q^{\frac{1}{2}}X_{n+k+1}+q^{-\frac{1}{2}}X_{n+k-1})+q^{-\frac{k}{2}}(q^{\frac{1}{2}}X_{n-k+1}+q^{-\frac{1}{2}}X_{n-k-1})\nonumber\\
&&-(q^{\frac{k-1}{2}}X_{n+k-1}+q^{-\frac{k-1}{2}}X_{n-k+1})\nonumber\\
&=&q^{\frac{k+1}{2}}X_{n+k+1}+q^{-\frac{k+1}{2}}X_{n-k-1}.\nonumber
\end{eqnarray}
This proves (2). Now we prove (3). If $m=0,$ it is obvious. If
$m=1,$ by the recurrence relations \eqref{recurrence}, we have
$$X_{n}X_{n+2}=qX^{2}_{n+1}+1.$$ We have proved that the equation
 $X_0X_3=qX_1X_2+q^{-\frac{1}{2}}z$ (see
\eqref{imaginary}) holds, thus by Lemma \ref{2}, we have
$$X_{n}X_{n+3}=qX_{n+1}X_{n+2}+q^{-\frac{1}{2}}z.$$
Now we assume that  equations in (3) hold for $m\leq k$. For
$m=k+1,$ by Lemma \ref{3}, we have
$$
X_{n}X_{n+2k+2}=
q^{-\frac{1}{2}}X_{n}(X_{n+2k+1}z_{1}-q^{-\frac{1}{2}}X_{n+2k}).
$$
Following the inductive assumption, it is equal to
$$
q^{-\frac{1}{2}}(q^{k}X_{n+k}X_{n+k+1}+\sum_{l=0}^{k-1}q^{-k+2l+\frac{1}{2}}\sum_{i=l+1}^k
z_{2(k-i)+1})z_{1}-q^{-1}(q^{k}X^{2}_{n+k}+\sum_{l=0}^{k-1}q^{-k+2l+1}\sum_{i=l+1}^k
z_{2(k-i)}).
$$
Using Lemma \ref{3} again and (1) of this proposition, it is
$$
q^{k-\frac{1}{2}}X_{n+k}(q^{\frac{1}{2}}X_{n+k+2}+q^{-\frac{1}{2}}X_{n+k})+\sum_{l=0}^{k-1}q^{-k+2l}\sum_{i=l+1}^k
z_{2(k-i)+1}z_{1}$$$$-q^{-1}(q^{k}X^{2}_{n+k}+\sum_{l=0}^{k-1}q^{-k+2l+1}\sum_{i=l+1}^k
z_{2(k-i)})
$$
$$
=q^{k+1}X^{2}_{n+k+1}+\sum_{l=0}^{k}q^{-k+2l}\sum_{i=l+1}^{k+1}
z_{2(k-i+1)}. \quad \quad \text{(*)}
$$
Similarly, by Lemma \ref{3}, we have
$$
X_{n}X_{n+2k+3} =
q^{-\frac{1}{2}}X_{n}(X_{n+2k+2}z_{1}-q^{-\frac{1}{2}}X_{n+2k+1})
$$
Using the equation (*) and similar proof, we obtain
$$
X_{n}X_{n+2k+3}
=q^{k+1}X_{n+k+1}X_{n+k+2}+\sum_{l=0}^{k}q^{-k+2l-\frac{1}{2}}\sum_{i=l+1}^{k+1}
z_{2(k-i+1)+1}.
$$
\end{proof}
\begin{remark}\label{8}
By Lemma \ref{4} and properties of bar-invariant, we can easily
obtain the similar results for $z_{m}X_{n}, X_{n+2m}X_{n},
X_{n+2m}X_{n}$.
\end{remark}
We similarly define the quantized version of the definition of
positivity in \cite{sherzel}.
\begin{definition}
A nonzero element $x\in\Acal_q(2,2)$ is positive if for every
$m\in\mathbb{Z},$ all the coefficients in the expansion of $x$ as a
Laurent polynomial in $\{x_m,x_{m+1}\}$ belong to
$\mathbb{N}[q^{\pm\frac{1}{2}}].$
\end{definition}
\begin{corollary}
Every element in  $\mathcal{B},\mathcal{S}$ and $\mathcal{D}$ is a
positive element of quantum cluster algebra $\Acal_q(2,2)$.
\end{corollary}
\begin{proof}
By Lemma \ref{1} and the fact $z_{n}(x)=s_{n}(x)-s_{n-2}(x)$, we
only need to prove every element in  $\mathcal{B}$ is positive. By
the definition of  $\sigma_{1}$ and Remark \ref{9}, it is enough to
prove the positivity in $\{x_1,x_{2}\}.$  We prove it by induction.
For convenience, we write down the following equations according to
Proposition \ref{5} and Remark \ref{8}:\\
 For $m\geq 1:$
\begin{eqnarray}
    X_{1}z_{m} &=&  q^{\frac{m}{2}}X_{1+m}+q^{-\frac{m}{2}}X_{1-m}.
\end{eqnarray}
For $m\geq 0:$
\begin{eqnarray}
   X_{1}X_{1+2m} &=& q^{m}X^{2}_{1+m}+q^{-m+1}\sum_{k=1}^{m}{z_{2(m-k)}}+q^{-m+3}\sum_{k=2}^{m}{z_{2(m-k)}}\\
   &&+\cdots+q^{m-3}\sum_{k=m-1}^{m}{z_{2(m-k)}}+q^{m-1};\nonumber\\
   X_{1}X_{2+2m} &=& q^{m}X_{1+m}X_{2+m}+q^{-m+\frac{1}{2}}\sum_{k=1}^{m}{z_{2(m-k)+1}}+q^{-m+\frac{5}{2}}\sum_{k=2}^{m}{z_{2(m-k)+1}} \\
&&+\cdots
+q^{m-\frac{7}{2}}\sum_{k=m-1}^{m}{z_{2(m-k)+1}}+q^{m-\frac{3}{2}}z_1.\nonumber
\end{eqnarray}

\begin{eqnarray}
   X_{1}X_{1-2m} &=& q^{-m}X^{2}_{1-m}+q^{m-1}\sum_{k=1}^{m}{z_{2(m-k)}}+q^{m-3}\sum_{k=2}^{m}{z_{2(m-k)}}\\
   &&+\cdots+q^{-m+3}\sum_{k=m-1}^{m}{z_{2(m-k)}}+q^{-m+1};\nonumber\\
   X_{1}X_{-2m} &=& q^{-m-1}X_{-m}X_{1-m}+q^{m-\frac{1}{2}}\sum_{k=1}^{m}{z_{2(m-k)+1}}+q^{m-\frac{5}{2}}\sum_{k=2}^{m}{z_{2(m-k)+1}} \\
&&+\cdots
+q^{-m+\frac{7}{2}}\sum_{k=m-1}^{m}{z_{2(m-k)+1}}+q^{-m+\frac{3}{2}}z_1.\nonumber
\end{eqnarray}
It is easy to check that
$$\{X_{-2},X_{-1},X_{0},X_{1},X_{2},z_1,z_2\}$$ are positive
elements in $\{x_1,x_{2}\}.$ Now assume that
$$\{X_{-2m},X_{-2m+1},\cdots,X_{2m-1},X_{2m},z_{1},\cdots,z_{2m}\}$$ are
positive elements in $\{x_1,x_{2}\}.$ Then by (3.2), (3.3) and
(3.4), we know that $X_{2m+1},X_{2m+2}$  and $X_{-1-2m}$ are
positive. Thus we obtain that $z_{2m+1}$ is positive by (3.1).
Therefore by (3.5), we have that $X_{-2m-2}$ is positive. Again by
(3.2), we know that $X_{2m+3}$ is positive. Thus we get $z_{2m+2}$
is positive by (3.1) again. Throughout the above discussions we
obtain
$$\{X_{-2m-2},X_{-2m-1},\cdots,X_{2m+1},X_{2m+2},z_{1},\cdots,z_{2m+2}\}$$
are positive elements in $\{x_1,x_{2}\}.$ The proof is finished.
\end{proof}
\begin{remark}
In fact, by \cite{Szanto}\cite{lampe}\cite{rupel}, the positivity in
the cluster variables is obvious, then applying the equation $
X_{1}z_{m}=q^{\frac{m}{2}}X_{1+m}+q^{-\frac{m}{2}}X_{1-m}$, we can
deduce the positivity in the elements $z_{m}$ for any
$m\in\mathbb{N}$. Here, we give an alternative proof without needing
the explicit expansions of cluster variables.
\end{remark}

\begin{theorem}
The sets $\mathcal{B},\mathcal{S}$ and $\mathcal{D}$ are
$\mathbb{Z}[q^{\pm \frac{1}{2}}]-$bases of the quantum cluster
algebra $\Acal_q(2,2)$.
\end{theorem}
\begin{proof}
Note that if $\mathcal{B}$ is a $\mathbb{Z}[q^{\pm
\frac{1}{2}}]-$basis of the quantum cluster algebra $\Acal_q(2,2)$,
then $\mathcal{S}$ and $\mathcal{D}$ are naturally
$\mathbb{Z}[q^{\pm \frac{1}{2}}]-$bases of quantum cluster algebra
$\Acal_q(2,2)$ because there are have unipotent matrix
transformations between $\{z_n\mid n\in \mathbb{N}\}$, $\{s_n\mid
n\in \mathbb{N}\}$ and $\{z^n\mid n\in \mathbb{N}\}$. In the
following, we will focus on the set $\mathcal{B}$ and prove it is a
$\mathbb{Z}[q^{\pm \frac{1}{2}}]-$basis of the quantum cluster
algebra $\Acal_q(2,2)$.

 By
Proposition \ref{5}, we obtain that any element  of the quantum
cluster algebra $\Acal_q(2,2)$ can be a $\mathbb{Z}[q^{\pm
\frac{1}{2}}]-$combination of the elemnets in the set $\mathcal{B}$.
Thus we only need to prove the elemnets in  $\mathcal{B}$ are
$\mathbb{Z}[q^{\pm \frac{1}{2}}]-$independent.

By Remark \ref{7}, we know that the elements $\textbf{c}$ associated
to  these minimal non-zero terms $f(q^{\frac{1}{2}},
q^{-\frac{1}{2}})X^{\textbf{c}}$ in the expansion of the elements in
the set $\mathcal{B}$ are different from  each other. Now we suppose
that a finite $\mathbb{Z}[q^{\pm \frac{1}{2}}]-$combination of the
elemnets in the set $\mathcal{B}$ is equal to $0.$ Let $S\subset
\mathbb{Z}^{2}$ be the set of all $\alpha$ such that the
corresponding element occurs with a non-zero coefficient in this
$\mathbb{Z}[q^{\pm \frac{1}{2}}]-$combination. If $S$ is non-empty,
pick a minimal element $\alpha\in S$, by Remark \ref{6} and Remark
\ref{7}, we know that $X^{\alpha}$ does not occur in the expansion
of any other element in above equation which gives a contradiction.
This completes the proof of the theorem.
\end{proof}

Set
$$\mathcal{B'}=\{q^{-\frac{ab}{2}}X^{a}_{m}X^{b}_{m+1}|m\in \mathbb{Z}, (a,b)\in \mathbb{Z}^{2}_{\geq
0}\}\cup \{z_{n}|n\in \mathbb{N}\}$$
$$\mathcal{S'}=\{q^{-\frac{ab}{2}}X^{a}_{m}X^{b}_{m+1}|m\in \mathbb{Z},(a,b)\in
\mathbb{Z}^{2}_{\geq 0}\}\cup \{s_{n}|n\in \mathbb{N}\}$$
$$\mathcal{D'}=\{q^{-\frac{ab}{2}}X^{a}_{m}X^{b}_{m+1}|m\in \mathbb{Z},(a,b)\in
\mathbb{Z}^{2}_{\geq 0}\}\cup \{z^{n}|n\in \mathbb{N}\}.$$ Then we
can obtain the following corollary.
\begin{corollary}
The sets $\mathcal{B'},\mathcal{S'}$ and $\mathcal{D'}$ are
bar-invariant $\mathbb{Z}[q^{\pm \frac{1}{2}}]-$bases of the quantum
cluster algebra $\Acal_q(2,2)$.
\end{corollary}

\section{An representation-theoretic interpretation of the element $s_n$}
 Recall we denote by $R_p(n)$ the indecomposable regular
modules with dimension vector $(n, n)$ for $n\geq 1$ and some $p\in
\mathbb{P}^1_k$ of degree 1. In this section, we will prove that
$s_n$ is equal to $X_{n\delta}$ for every $n \in \mathbb{N}.$
 The following proposition shows the Laurent expansion of $X_m$ in
 $A_q(2,2)$.
\begin{proposition}(\cite{lampe},\cite{rupel})\label{11}
For every $n \geq 0$, we have
\begin{eqnarray}
&\label{eq:X-n-formula}X_{-n} &\textstyle  =  X^{(n+2,-n-1)} +
\sum_{p + r \leq
n} {n-r \brack p}_q{n+1-p \brack r}_q X^{(2r-n,2p-n-1)};\\
&\label{eq:Xn-formula}X_{n+3} &\textstyle  =  X^{(-n-1,n+2)} +
\sum_{p + r \leq n} {n-r \brack p}_q{n+1-p \brack r}_q
X^{(2p-n-1,2r-n)}.
\end{eqnarray}
\end{proposition}

\begin{lemma}\label{10}
For  every $n \in \mathbb{N}$, we have
$$s_n=q^{\frac{n}{2}}X_{1}X_{n+3}-q^{\frac{n}{2}+1}X_{2}X_{n+2}.$$
\end{lemma}
\begin{proof}
It is easy to check that
$s_1=q^{\frac{1}{2}}X_{1}X_{4}-q^{\frac{3}{2}}X_{2}X_{3}.$ Assume
that it holds for $n\leq k,$ then
\begin{eqnarray}
  s_{k+1}  &=& s_ks_1-s_{k-1}\nonumber\\
  &=& (q^{\frac{k}{2}}X_{1}X_{k+3}-q^{\frac{k}{2}+1}X_{2}X_{k+2})s_1-(q^{\frac{k-1}{2}}X_{1}X_{k+2}-q^{\frac{k+1}{2}}X_{2}X_{k+1})\nonumber\\
 &=& q^{\frac{k}{2}}X_{1}(q^{\frac{1}{2}}X_{k+4}+q^{-\frac{1}{2}}X_{k+2})-q^{\frac{k}{2}+1}X_{2}(q^{\frac{1}{2}}X_{k+3}+q^{-\frac{1}{2}}X_{k+1}) \\
&-&(q^{\frac{k-1}{2}}X_{1}X_{k+2}-q^{\frac{k+1}{2}}X_{2}X_{k+1})\nonumber\\
 &=&   q^{\frac{k+1}{2}}X_{1}X_{k+4}-q^{\frac{k+1}{2}+1}X_{2}X_{k+3}.\nonumber
\end{eqnarray}

\end{proof}

Denote the quantum binomial coefficients ${n\choose
r}_q=\frac{(q^n-1)(q^{n-1}-1)\cdots(q^{n-r+1}-1)}{(q^r-1)\cdots(q-1)}$
and take ${n\choose0}_q=1$ for any $n\in\ZZ$, ${n\choose l}_q=0$ for
any $l<0,$ and ${n\choose l}_q=0$ for any $0\leq n<l.$ Then we have
the following theorem proved in \cite{Szanto}:

\begin{theorem}\cite[Theorem 4.6]{Szanto}\label{12}
Let $\textbf{e}=(a, b)$ for $(a, b)\in \mathbb{Z}^2_{\geq 0}$. Then
for $n\geq 1,$ $$|Gr_{\textbf{e}}(R_p(n))|={n-a \choose
n-b}_{q^2}{b\choose a}_{q^2}.$$
\end{theorem}

\begin{proposition}
$X_{R_p(n)}=s_n$ for $n \in \mathbb{N}$.
\end{proposition}
\begin{proof}
By the $q$-deformation of the Caldero-Chapoton formula and Theorem
\ref{12}, we have
\begin{align*}
X_{R_p(n)} &=\sum_{(a,b)}q^{-\half d_{\bf (a,b)}^{R_p(n)}}|Gr_{(a,b)}(R_p(n))|X^{(n-2b,2a-n)}\\
&=\sum_{(a,b)}q^{(a-b)(n-b)}{n-a\choose
n-b}_{q^2}q^{(a-b)a}{b\choose
a}_{q^2}X^{(n-2b,2a-n)}\\
& =\sum_{(a,b)}{n-a\brack n-b}_q{b\brack a}_qX^{(n-2b,2a-n)}\\
&=\sum_{p + r \leq n}{n-r\brack p}_q{n-p\brack r}_qX^{(2p-n,2r-n)}.
\end{align*}

On the other hand, by Proposition \ref{11}, we have
\begin{align*}
q^{\frac{n}{2}}X_{1}X_{n+3} &=q^{\frac{n}{2}}X_{1}X^{(-n-1,n+2)}
+q^{\frac{n}{2}}X_{1} \sum_{p + r \leq n} {n-r \brack p}_q{n+1-p
\brack r}_q
X^{(2p-n-1,2r-n)}\\
&=q^{n+1}X^{(-n,n+2)} + \sum_{p + r \leq n} {n-r \brack p}_q{n+1-p
\brack r}_q q^{r}X^{(2p-n,2r-n)}.
\end{align*}
And
\begin{align*}
q^{\frac{n}{2}+1}X_{2}X_{n+2} &=q^{\frac{n}{2}+1}X_{2}X^{(-n,n+1)}
+q^{\frac{n}{2}+1}X_{2} \sum_{p + r \leq n-1} {n-1-r \brack p}_q{n-p
\brack r}_q
X^{(2p-n,2r+1-n)}\\
&=q^{n+1}X^{(-n,n+2)} +\sum_{p + r \leq n-1} {n-1-r \brack p}_q{n-p
\brack r}_q
q^{n-p+1}X^{(2p-n,2r+2-n)}\\
&=q^{n+1}X^{(-n,n+2)} +\sum_{p + r \leq n} {n-r \brack p}_q{n-p
\brack r-1}_q q^{n-p+1}X^{(2p-n,2r-n)}.
\end{align*}

Note that it is easy to confirm the following identity:
$${n+1-p \brack r}_q q^{r}={n-p \brack r}_q+{n-p \brack r-1}_q
q^{n-p+1}.$$ Then by Lemma \ref{10}, we have
\begin{eqnarray}
  s_n  &=& q^{\frac{n}{2}}X_{1}X_{n+3}-q^{\frac{n}{2}+1}X_{2}X_{n+2}\nonumber\\
  &=& \sum_{p + r \leq n}{n-r\brack p}_q{n-p\brack
r}_qX^{(2p-n,2r-n)}\nonumber\\
 &=& X_{R_p(n)}.\nonumber
\end{eqnarray}

\end{proof}

\section*{Acknowledgements}
The authors would like to thank Professor Jie Xiao for many helpful
discussions.


\end{document}